\documentclass[letterpaper,12pt]{amsart}

\usepackage{url}
\usepackage{graphicx}
\usepackage{graphics}
\usepackage{amsmath}
\usepackage{float}
\usepackage{amssymb,amsmath}
\usepackage{enumerate}
\usepackage{hyperref}

\usepackage{geometry}
\theoremstyle{plain}
\newtheorem{theorem}{Theorem}[section]

\newtheorem{corollary}[theorem]{Corollary}

\theoremstyle{remark}

\newcommand{\cp}{\ell^1(G,A,\alpha)}

\newcommand{\Cstar}{\ensuremath{{\mathrm C}^\ast}}

\title{Hermitian crossed product Banach algebras}

\author{R. El Harti}
\address{Department of Mathematics and Computer Sciences,
Faculty of Sciences and Techniques,
	University Hassan I, BP 577, Settat, Morocco}

\author{Paulo R.\ Pinto}
\address{Department of Mathematics, CAMGSD, Instituto Superior T\'{e}cnico, Univ.\ Lisboa,
Av.\ Rovisco Pais 1, 1049-001 Lisbon, Portugal.}

\date{\today}

\begin{document}

\begin{abstract}
We show that the Banach *-algebra $\ell^1(G,A,\alpha)$, arising from a C*-dynamical system $(A,G,\alpha)$, is an hermitian Banach algebra if the discrete group $G$ is finite or abelian (or more generally, a finite extension of a nilpotent group).

As a corollary, we obtain that $\ell^1(\mathbb{Z},C(X),\alpha)$ is hermitian, for every topological dynamical system $\Sigma = (X, \sigma)$, where $\sigma: X\to X$ is a homeomorphism of a compact Hausdorff space $X$ and the action is $\alpha_n(f)=f\circ \sigma^{-n}$ with $n\in\mathbb{Z}$.

\end{abstract}

\maketitle

\medskip


\medskip

MSC\ {46L65, 43A20, 54H20, 46K99}

\tableofcontents

\section{Introduction}

Given a \Cstar-algebra $A$ and $\alpha: G\to \hbox{Aut}(A)$ an action of a discrete group $G$ on $A$, then an involutive Banach algebra $\cp$ of crossed product type can be associated with this  \Cstar-dynamical system $(A,G,\alpha)$. A fundamental class of such examples $\ell^1(\Sigma)=\ell^1(\mathbb{Z},C(X),\alpha)$ arise from topological dynamical systems $\Sigma=(X,\sigma)$, with $X$ a compact and Hausdorff spaces and $\sigma: X\to X$ homeomorphisms, so that we have the \Cstar-dynamical system $(C(X), \mathbb{Z}, \alpha)$, with $\alpha_n(f)=f\circ \sigma^{-n}$ and $n\in\mathbb{Z}$, see \cite{marcel1, tomi}.

A non-C*-problem is to decide when $\cp$ is a hermitian Banach algebra (i.e.\ the spectrum of every self-adjoint element is real). It is an even more appealing and intriguing problem when we restrict to those $\ell^1(\Sigma)$ arising from topological dynamics, as it was actually proposed in \cite{marcel00}. In fact in \cite{marcel00}, the answer was find to be positive if $X$ is a finite set, and then \cite{marcel1} the result was generalized for the cases where all points of $X$ are periodic

The objective of this paper is to prove that $\ell^1(\Sigma)$ is a hermitian Banach algebra, thus answering a question of \cite{marcel00,marcel1}. Moreover, our proof also works in the more general context of a \Cstar-dynamical system $(A,G,\alpha)$ provided is $G$ is finite or a finite extension of nilpotent group (for example, abelian groups) and $A$ is a unital \Cstar-algebra. In order to prove this result, we first isometrically embed
$\cp$ into  $\ell^1(G,\widehat{A},\widehat{\alpha})$, where $\widehat{A}$ is a unital \Cstar-algebra that contains $G$ and $A$ and the action $\widehat{\alpha}$ is inner. Since the latter action is inner, $\ell^1(G,\widehat{A},\widehat{\alpha})$ is proved to be isometrically *-isomorphic to the projective tensor product $\ell^1(G)\ \widehat{\otimes}\ \widehat{A}$. Thus we obtain an embedding (isometric *-homomorphism) of Banach *-algebras
\begin{equation}\label{eq001}
	\cp\, \subset\, \ell^1(G)\ \widehat{\otimes}\ \widehat{A}
\end{equation}
thus our examples of hermitian $\ell^1$ Banach algebras arise from the ones $\ell^1(G)\, \widehat{\otimes}\, \widehat{A}$ that are hermitian.

The plan for the rest of the paper is as follows. In Sect.\ \ref{preliminaries}, we review some elements of the theory of hermitian Banach algebras. Since the *-Banach algebra $\cp$ are the main focus of this paper, we end this section with a somehow detailed definition of it \cite{hartimarcelpinto2,marcel1,tomi}.

In Sect.\ \ref{mainresults}, we first show that if we are in presence of a \Cstar-dynamical system $(B,G,\beta)$ with an inner action $\beta$, then  $\ell^1(G,B,\beta)$ does not depend on the action, namely, $\ell^1(G,B,\beta)$ is isometrically *-isomorphic to $\ell^1(G,B,\hbox{triv})$, with trivial action, which in turn is also isometrically *-isomorphic to the projective tensor product $\ell^1(G)\, \widehat{\otimes}\, B$, see Theorem \ref{lemma1}. Then Theorem \ref{lemma2} shows that for a generic \Cstar-dynamical system $(A,G,\alpha)$, the Banach algebra $\cp$ sits inside in $\ell^1(G,\widehat{A},\widehat{\alpha})$ for a certain \Cstar-algebra $\widehat{A}$ and an inner action $\widehat{\alpha}$. For this to work we need to consider the natural homomorphism $\pi: A\to \widehat{A}$ and then check that $\pi$ is equivarant with respect to actions, namely,
 $$\widehat{\alpha}_g\circ \pi = \pi\circ \alpha_g\ \hbox{for all}\ g\in G.$$
Finally, we we consider a generic \Cstar-dynamical system $(A,G,\alpha)$, we can use the aforementioned lemmas to yield an embedding (isometric *-homomorphism) as in \eqref{eq001} because as Banach *-algebras we have:
$$\cp\, \subset\, \ell^1(G,\widehat{A},\widehat{\alpha}) \, \simeq \, \ell^1(G,\widehat{A},\hbox{triv}) \, \simeq \, \ell^1(G)\, \widehat{\otimes}\, \widehat{A}.$$
 Because the hermitian property passes to Banach *-subalgebras, we look for discrete groups $G$ for which $\ell^1(G)\, \widehat{\otimes}\, B$ is hermitian with $B$ a \Cstar-algebra. This is indeed the notion of $G$ being rigidly symmetric (i.e., $\ell^1(G)\, \widehat{\otimes}\, B$ is hermitian with $B$ a \Cstar-algebra) that appeared in \cite{leptin}. This then leads to the main Theorem \ref{mainthm}, thus given a discrete group $G$ and a unital \Cstar-algebra $A$, then  $\ell^1(G,A,\alpha)$ is hermitian when $G$ is rigidly symmetric.

As in \cite{leptin}, if $G$ is abelian or finte or if $G$ has a central subgroup $N$ such that $G/N$ is finite or nilpotent, then $\ell^1(G,A,\alpha)$ is hermitian.

Finally, thanks to Theorem \ref{mainthm}, we can now claim that $\ell^1(\Sigma)$ is hermitian, as in Corollary \ref{corfinal}, as the group $G$ involved is the abelian group $G=\mathbb{Z}$.

%


\section{Preliminaries}
\label{preliminaries}

As already mentioned in the introduction, a complex Banach *-algebra $A$ (i.e. an involutive Banach algebra) is said to be an hermitian algebra if the spectrum sp$_A(a)$ of every selfadjoint $a\in A$ is a subset of the real
numbers sp$_A(a)\subset \mathbb{R}$. There are several characterizations for a Banach algebra to be hermitian, see for example \cite{palmer2}. Let us just mention one, which is called the Shirali-Ford’s Theorem (see \cite{shirali}) that reads that $A$ is hermitian if and only if $A$ is symmetric (i.e.\  sp$_A(aa^\ast)\subset \mathbb{R}_0^+$ for all $a\in A$). We remark that n fact, Shirali-Ford’s Theorem is about the implication hermitean implies symmetric, and the other implication is much easier. see e.g.\ \cite[Theorem 11.4.1]{palmer2}.

If we have a Banach *-subalgebra $A\subset B$  of a Banach *-algebra $B$, then for every $a\in A$, $\partial\hbox{sp}_A (a) \subset \partial\hbox{sp}_B (a)$ and so $A$ inherits the hermitian property from $B$.

\Cstar-algebras are all hermitian Banach algebras.  For simplicity, we will say that a group $G$ is hermitian if the group Banach algebra $\ell^1(G)$ is hermitian. Then finite groups or abelian groups are all hermitian and, moreover, nilpotent groups are also hermitian, see \cite{palmer2}. It is well known that the free groups $F_n$ are not hermitian as well the disc algebra $A(D)$. It has been known since the 60's that amenable groups might fail be hermitian, see \cite{jenkins} for one example. Note that, it has been recently shown that hermitian groups are amenable, see \cite{samei}.

 Every element of the algebraic tensor product $A\otimes B$ may be written (in many different ways) as a (finite) linear combination of tensor products $\sum a_i \otimes b_i$. The projective norm $||x||_\pi$ of an element $x\in A\otimes B$ is

$$||x||_\pi= \inf \{\sum ||a_i||\ ||b_i||: \ x=\sum a_i\otimes b_i\}.$$

Then the projective tensor product $A\, \widehat{\otimes}\, B$ is the completion of $A \otimes B$ under the projective norm $||\cdot||_\pi$. We remark that the projective norm is a cross norm, i.e.\ $||a\, \otimes \, b||_\pi=||a||_A\, ||b||_B$.

A way to fabricate larger hermitian algebra is when we consider the projective tensor product.
Thus, the projective tensor product $A\, \widehat{\otimes}\, B$ between hermitian Banach algebras is hermitian provided one of them is abelian. This is a consequence of \cite[Corollary 3.3]{bonic}, where this stability result is proved among symmetric algebras, however, symmetric and hermitian Banach algebra are the same by the aforementioned Shirali-Ford’s Theorem \cite{shirali}. As a way of examples, if $G$ and $H$ are hermitian and one of them is abelian then the direct product $G\times H$ is hermitian. This is a consequence of Grothendieck's result $\ell^1(G\times H) \simeq \ell^1(G)\, \widehat{\otimes}\, \ell^1(H)$. If one of the algebras in the tensor product is a \Cstar-algebra, then one can yield more examples of of hermitian algebras, namely, $B$ is a \Cstar-algebra, then $\ell^1(G)\, \widehat{\otimes}\, B$ is hermitian if $G$ is a finite extension of a nilpotent group, see \cite{leptin}.

Since the focus of this paper are the Banach *-algebras of the form $\cp$ and for the benefit of the reader, we now define this Banach algebra, which is quite standard, as follows. Suppose an action $\alpha: G\to \rm{Aut}(A)$ of the discrete group $G$ as $^\ast$-automorphisms of a \Cstar-algebra $A$ is given, so that we have a \Cstar-dynamical system $(A,G,\alpha)$. Let $\Vert \cdot \Vert $ denote the norm of the \Cstar-algebra $A$, and let
\begin{equation}\label{eq:definition_as_maps}
	\cp=\{\texttt{a}: G\longrightarrow A:\ \Vert \texttt{a}\Vert:=\sum_{g\in G} \Vert a_g\Vert <\infty\,\},
\end{equation}
where we have written $a_g$ for $\texttt{a}(g)$.
We supply $\cp$ with the usual twisted convolution product and involution, defined by
\begin{equation}\label{eq:prod}
	(\texttt{a}\texttt{a}^\prime)(g)=\sum_{k\in G} a_k \cdot \alpha_k(a_{k^{-1}g}^\prime)\quad\hbox{with}\quad g\in G\ \hbox{and} \ \texttt{a},\texttt{a}^\prime\in \cp
\end{equation}
and
\begin{equation}\label{eq:involution}
	\quad \texttt{a}^\ast(g)=\alpha_g((a_{g^{-1}})^\ast) \quad\hbox{with}\quad g\in G \ \hbox{and}\ \texttt{a}\in \cp,
\end{equation}
respectively, so that $\cp$ becomes an involutive Banach algebra. If $A=\mathbb C$, then $\ell^1(G,\mathbb C,\textrm{triv})$ is the usual group algebra $\ell^1(G)$.
If $G=\mathbb Z$ and $A$ is a commutative unital \Cstar-algebra, hence of the form $\textrm{C}(X)$ for a compact Hausdorff space $X$, then there exists a homeomorphism $\sigma: X\to X$ such that the $\mathbb Z$-action  is given by $\alpha_n(f)=f\circ \sigma^{-n}$ for all $f\in\textrm{C}(X)$ and $n\in\mathbb Z$. See \cite{marcel1}, for recent developments of the  Banach algebra $\ell^1(\mathbb Z,\textrm{C}(X), \alpha)$.

Let us return to the case of an arbitrary \Cstar-algebra $A$, and assume for that $A$ is unital. A convenient way to work with $\cp$ is then provided by the following observations.

For $g\in G$, let $\delta_g: G\to A$ be defined by
$$\delta_g(k)=\left\{
\begin{array}{ll}
	1_A & \hbox{if}\ k=g;\\
	0_A & \hbox{if}\ k\not=g,
\end{array} \right.$$
where $1_A$ and $0_A$ denote the unit and the zero element of $A$, respectively. Then $\delta_g\in \cp$, and $\Vert \delta_g\Vert =1$ for all $g\in G$. Furthermore, $\cp$ is unital with $\delta_e$ as the unit element, where $e$ denotes the identity element of the group $G$. Using \eqref{eq:prod}, one finds that \begin{equation*}
	\delta_{gk}=\delta_{g}\cdot \delta_k \quad
\end{equation*}
for all $g,k\in G$. Hence, for all $g\in G$, $\delta_g$ is invertible in $\cp$, and, in fact, $\delta_g^\ast=\delta_{g^{-1}}$. It is now obvious that the set $\{\,\delta_g : g\in G\,\}$ consists of norm one elements of $\cp$, and that it is a subgroup of the invertible elements of $\cp$ that is isomorphic to $G$.

In the same vein, it follows easily from \eqref{eq:prod} and \eqref{eq:involution} that we can view
$$A\subset \cp$$
as a closed *-subalgebra, namely as $\{\,a\delta_e: a\in A\,\}$, where $a\delta_e$ is the element of $\cp$ that assumes the value $a\in A$ at the identity $e\in G$, and the value $0_A\in A$ elsewhere.

If $\texttt{a}\in \cp$, then it is easy to see that $\texttt{a}=\sum_{g\in G} (a_g \delta_e)\delta_g$ as an absolutely convergent series in $\cp$. Hence, if we identify $a_g \delta_e$ and $a_g$, we have $\texttt{a}=\sum_{g\in G} a_g \delta_g$ as an absolutely convergent series in $\cp$.

Finally, let us note that an elementary computation, using the identifications just mentioned, shows that the identity
\begin{equation}\label{eq:conjugation}
	\delta_g a\delta_g^{-1}=\alpha_g(a)
\end{equation}
holds in $\cp$ for all $g\in G$ and $a\in A$, see \cite{marcel1,hartimarcelpinto2}.

If $G=\mathbb Z$ and $A$ is a commutative unital \Cstar-algebra, hence of the form $\textrm{C}(X)$ for a compact Hausdorff space $X$, then there exists a homeomorphism of $X$ such that the $\mathbb Z$-action  is given by $\alpha_n(f)=f\circ \sigma^{-n}$ for all $f\in\textrm{C}(X)$ and $n\in\mathbb Z$.  The algebra $\ell^1(\mathbb Z,\textrm{C}(X), \alpha)$ has been studied in \cite{marcel1}, \cite{JTStudia}, and \cite{marcel1},

\section{Main results}
\label{mainresults}

The proof of the first isomorphism of the following result is an adaptation from the \Cstar-crossed product algebras.
\begin{theorem}\label{lemma1}
Let $\beta: G\to Aut(B)$ be an inner action of a discrete group $G$ on a \Cstar-algebra $B$. Then we have
\begin{equation*}
	\ell^1(G;B,\beta) \ \simeq\ \ell^1(G,B,\textrm{triv})\ \simeq\ \ell^1(G)\, \widehat{\otimes}\, B
\end{equation*}
as isometric *-isomorphisms
\end{theorem}
\begin{proof}
For every $g\in G$, let $u_b\in B$ be the implementing unitary of the automorphism $\beta_g$, so that $\beta_g(b)=u_bbu_b^\ast$ (as we assume that $\beta$ is inner). We also have the trivial action $\iota$ of $G$ on $A$, thus $\iota_g(b)=b$ for all $b\in B$ and $g\in G$.

Let $\delta_g\in \ell^1(G;B,\beta)$ and  let $v_g\in \ell^1(G,B,\textrm{triv})$ be the unitaries corresponding to the group elements. We thus have $\beta_g(b)=\delta_g b\delta_g^\ast$ and $v_g b=\iota_g(b)v_g=bv_g$.

 Then we consider the map $\varphi: \ell^1(G,B,\beta) \ \to\ \ell^1(G,B,\textrm{triv})$ such that
$$b\ \delta_g \mapsto bu_g\, v_g.$$
We check that $\varphi$ is a *-isometric bijective homomorphism. First, $\varphi$ is
clearly a linear map bijection between the (algebraic) group algebras, whose inverse maps $bu_g^\ast\delta_g\mapsto a\, v_g$. Then $\varphi$   is *-preserving
\begin{eqnarray*}
	\varphi((b\delta_g)^\ast)&=& \varphi(\beta_{g^{-1}}(b^\ast) \delta_{g^{-1}})
=\beta_{g^{-1}}(b^\ast) u_g^\ast\ v_{g^{-1}}= (u_g^{\ast}b^\ast u_g) u_g^\ast \ v_{g^{-1}}\\ &=& u_g^\ast b^\ast \ v_{g^{-1}}=(bu_g\ v_g)^\ast=(\varphi(b\delta_g))^\ast.\end{eqnarray*}

Now, we check that $\varphi$ is an algebra homomorphism:
\begin{eqnarray*}
	\varphi(a\delta_g)\varphi(b\delta_h)&=& (au_g\, v_g)(bu_h\, v_h)=au_g bu_g^\ast u_{gh}\, v_{gh})\\
	&=& a\beta_g(b) u_{gh}\, v_{gh}= \varphi(a\beta_g(b)\, \delta_{gh})=\varphi ((a\, \delta_g)(b\, \delta_h)).
\end{eqnarray*}
Next, $\varphi$ is isometric because of the definitions of the norms of  $\ell^1(G;B,\beta)$ and $\ell^1(G,B,\textrm{triv})$ together with the fact that $B$ is a \Cstar-algebras (where $||bu_g||=||b||$), as we can first check on finite sums and then extend it, namely if $x=\sum b_g\, \delta_g$ then
$$||\varphi(x)||= \sum ||b_gu_g|| = \sum ||b_g||= ||x||.$$

Finally, note that $\ell^1(G,B,\textrm{triv})\ \simeq\ \ell^1(G)\, \widehat{\otimes}\, B$ is well known, see e.g.\ \cite[Theorem 1.5.4]{Kaniuth}. In any case, for the sake of completeness, we check that the map $\phi: \ell^1(G,B,\textrm{triv})\ \to\ \ell^1(G)\, \widehat{\otimes}\, B$ such that
$$\sum_g b_g\, v_g\mapsto \sum d_g \widehat{\otimes}\, b_g$$ is an isometric *-isomorphism, where $d_g\in \ell^1(G)$ is the dirac funcion $d_g(h)=1$ if $h=g$ and 0 otherwise. Indeed, $\phi$ is an injective *-homomorphism and if $x=\sum b_g\ v_g$, then it is clear that
\begin{equation}\label{eqleq1}
	||\phi(x)||\leq ||x||.
\end{equation}
 For the surjectivity of $\phi$, let $y=\sum r_i\otimes b_i\in  \ell^1(G)\, \widehat{\otimes}\, B$. Remark now that for every $i$, $r_i=\sum_g \lambda_g^{(i)} d_g$ for certain complex numbers $\lambda_g^{(i)}$. So
$$y=\sum_i r_i \otimes b_i=\sum \sum \lambda_g^{(i)}d_g \otimes b_i=\sum\sum d_g\otimes \lambda_g^{(i)}b_i.$$
Thus, if we set $x:=\sum_g\sum_i \lambda_g^{(i)}b_i\, v_g\in \ell^1(G,B,\textrm{triv})$ then
$\phi(x)=y$, implying the surjectivity of $\phi$. Besides this,
\begin{eqnarray*}
	||x||_{\ell^1(G,B,\textrm{triv})} &=&\sum_g||\sum \lambda_g^{(i)}b_i||\ = \sum\sum |\lambda_g^{(i)}|\, ||b_i||= \sum_i ||r_i||\ ||b_i||
\end{eqnarray*}\\
for any representation of $y=\sum_i r_i\, \otimes b_i$. Thus by the definition of the projective norm we have $||x|| \leq ||\phi(x)||$,
which, together with \eqref{eqleq1}, lead us to conclude that $\phi$ is isometric.
\end{proof}

 Let $A\subset B(H)$ and let $\ell^2(G,H)$ be the Hilbert space of square summable $H$-valued functions on $G$.
Then define a unitary representation $\lambda$ of $G$ and a faithful *-representation $\pi$ of $A$ acting on the same Hilbert space $\ell^2(G,H)$, as follows:
\begin{equation}\label{eqlrep}
	(\lambda_g \xi) (t)=\xi(g^{-1}t),\quad \hbox{and}\ ((\pi(a) \xi) (t)=\alpha_{t^{-1}}(a) \xi (t),\ \ \hbox{with}\ \xi\in \ell^2(G,H),\ g,t\in G.
\end{equation}

\begin{theorem}\label{lemma2}
Let $\beta: G\to Aut(A)$ be an action of a discrete group $G$ on a \Cstar-algebra $A$. Then we there exists a \Cstar-algebra $\widehat{A}$ and a inner action $\widehat{\alpha}$ of $G$ on $\widehat{A}$ such that
$$\ell^1(G,A,\alpha) \ \subset\ \ell^1(G,\widehat{A},\widehat{\alpha}).$$
\end{theorem}
\begin{proof}
Since $A$ is a \Cstar-algebra, we can assume that $A\subset B(H)$ for some Hilbert space $H$.
Then we may use the representations $\pi$ and $\lambda$ defined in \eqref{eqlrep} and set $\widehat{A}=C^\ast(\pi(A), \lambda(G))$ and and the inner action $\widehat{\alpha}(b)=\lambda_g b \lambda_g^\ast$, with $b\in \widehat{A}$, of $G$ on $\widehat{A}$. We then can check the equivariance property of $\pi$:
\begin{equation}\label{eqeq1}
	\widehat{\alpha}_g\circ \pi=\pi \circ \alpha_g, \ \hbox{for all}\ g\in G,
\end{equation}
namely because
\begin{eqnarray*}
 (\lambda_g \pi(a)\lambda_g^\ast \xi) (t)&=&\pi(a) \lambda_g \xi(g^{-1}t)
=\alpha_{t^{-1}g}(a) (\lambda_{g^{-1}} \xi (g^{-1}t))\\
&=& \alpha_{t^{-1}}(\alpha_g(a)) \xi (t)= \pi(\alpha_g(a))\xi (t).
\end{eqnarray*}

We then define a map $\rho: \ell^1(G,A,\alpha)\to \ell^1(G,\widehat{A},\widehat{\alpha})$ such that $a\, \delta_g\mapsto \pi(a)\, \widehat{\delta}_g$, which is now easily checked to be an isometric *-homomorphism.
For example, we check that $\rho$ is a homomorphism, as for $a,b\in A$ and $g,h\in G$ we have
\begin{eqnarray*}
\rho((a\, \delta_g)(b\, \delta_h))&=&\rho(a\alpha_g(b)\, \delta_{gh})=\pi(a\alpha_g(b))\, \widehat{\delta}_{gh}\\
&=&\pi(a)\pi(\alpha_g(b))\, \widehat{\delta}_{gh}=\pi(a)\widehat{\alpha}_g(\pi(b))\, \widehat{\delta}_{gh} \\
&=& \pi(a)\, \widehat{\delta}_g\, \pi(b)\, \widehat{\delta}_h=\rho(a\, \delta_g)\, \rho(b\, \delta_h)
\end{eqnarray*}
where we use \eqref{eqeq1} in the 4th equality. Similarly $\rho$ is *-preserving as
\begin{eqnarray*}
(\rho(a\, \delta_g))^\ast&=&(\pi(a)\, \widehat{\delta_g})^\ast=\widehat{\alpha}_{g^{-1}}(\pi(a^\ast)) \widehat{\delta}_{g^{-1}}\\
&=&\pi(\alpha_{g^{-1}}(a^\ast)) \widehat{\delta}_{g^{-1}}=\rho(\alpha_{g^{-1}}(a^\ast)\, \delta_{g^{-1}})=\rho((a\, \delta_g)^\ast)
\end{eqnarray*}
where we use again \eqref{eqeq1} in the 3rd equality.
\end{proof}


We remark that if $\alpha$ is an inner action, then Theorem \ref{lemma1} implies that $\ell^1(G,A,\alpha)$ is hermitian, provided  $G$ abelian because the projective tensor product between hermitian Banach algebras is still hermitian (provided one of them is abelian), see \cite{bonic}. For a generic action $\alpha$ and making profit of the $\widehat{A}$ is a \Cstar-algebra, and the notion of $G$ being rigidly symmetric (i.e., $\ell^1(G)\, \widehat{\otimes}\, B$ is hermitian with $B$ a \Cstar-algebra) as in \cite[Page 132]{leptin}, we obtain the following result.

\begin{theorem}\label{mainthm}
Let $(A,G,\alpha)$ be a \Cstar-dynamical system, where $A$ is unital and $G$ discrete group.
Then $\cp$ is an hermitian Banach algebra if $G$ is rigidly symmetric.

As examples of such discrete groups $G$ are as follows:\\
(1) finite or \\
(2) abelian or more generally\\
(3) a finite extension of a nilpotent group.
\end{theorem}
\begin{proof}
	Thanks to Theorems \ref{lemma1} and \ref{lemma2}, $\cp \subset \ell^1(G)\, \widehat{\otimes}\, \widehat{A}$ where the embedding is given by an isometric *-homomorphism, thus $\cp$ is closed in $\ell^1(G)\, \widehat{\otimes}\, \widehat{A}$.
	Now note that the property of being hermitian passes to Banach *-subalgebras, so the first part of the theorem follows by the definition of $G$ being rigidly symmetric.
	
	Those examples of rigidly symmetric groups can be retrieved from \cite[Corollary 3]{leptin}. This finishes the proof of the theorem.
\end{proof}

Let $X$ be a compact Hausdorff space,  $\sigma: X\to  X$ a homeomorphism and $(C(X),\mathbb{Z},\alpha)$ the associated \Cstar-dynamical system, where $\alpha_n(f)=f\circ \sigma^{-n}$ with $n\in\mathbb{Z}$.

It is proved in \cite[Theorem 4.7]{marcel1} that $\ell^1(\mathbb{Z}, C(X),\alpha)$ is hermitian if all the points of $X$ are periodic. We are now in ready to generalize this results as follows.

\begin{corollary}\label{corfinal}
For any topological dynamical systema $\Sigma=(X, \sigma)$, we have that  $\ell^1(\mathbb{Z}, C(X),\alpha)$  is an hermitian Banach algebra.
\end{corollary}
\begin{proof}
It is an immediate consequence of Theorem \ref{mainthm}.
\end{proof}

We remark that Corollary \ref{corfinal} can be derived from a more primitive fact that the protective tensor product of two hermitian algebras is hermitian again, provided that one of the two factors is commutative (together with the above Theorems \ref{lemma1} and \ref{lemma2}).

We can provide more examples of hermitian crossed product Banach algebras, as an application of Theorem \ref{mainthm}. For example:
\begin{enumerate}
	
	\item Let $\mathcal{O}_n$ be the Cuntz algebra with generators $s_1,...,s_n$. We have the action of $\mathbb{Z}_n$ on $\mathcal{O}_n$ such that  $\alpha_r(s_i)=s_{i+r\, (\hbox{\scriptsize mod}\ n)}$, with $r\in\{0,...,n-1\}$.	
	
Then $\ell^1(\mathbb{Z}_n, \mathcal{O}_n, \alpha)$ is an hermitian Banach algebra.

	
	\item The Heisenberg group $\mathbb{H}$ ($3\times 3$ upper triangular matrices with entries in $\mathbb{Z}$ and diagonal constant equals 1) is an hermitian group because  $\mathbb{H}$ is nilpotent. The group  $\mathbb{H}$ acts on itself $g\cdot h=g^{-1}h$ giving rise to an action on the full group C*-algebra $C^\ast(\mathbb{H})$ by
	$$\alpha_g\left(\sum_h c_h\, \delta_h\right)=\sum_h c_h\, \delta_{g^{-1}h},\quad \hbox{where}\ \sum_h c_h\, \delta_h\in C^\ast(\mathbb{H}).$$
	Then Theorem \ref{mainthm} implies that $\ell^1(\mathbb{H}, C^\ast(\mathbb{H}), \alpha)$ is hermitian. Of course any discrete nilpotent group can likewise produce an hermitian Banach algebra.
\end{enumerate}

\section*{Acknowledgements}
The authors thank Lenny Marcel de Jeu for helpful comments and suggestions.
The second author was partially supported by FCT/Portugal through CAMGSD, IST-ID, projects UIDB/04459/2020 and UIDP/04459/2020.



\end{document}